\documentclass[12pt,twoside,reqno]{article}

\usepackage{amsfonts, amsthm, amsmath, amssymb}
\usepackage{hyperref}
\hypersetup{colorlinks=false}

\usepackage[margin=1.5in]{geometry}
\RequirePackage{mathrsfs} \let\mathcal\mathscr
\numberwithin{equation}{section}

\usepackage[numbers,sort]{natbib}
\allowdisplaybreaks[4]

\newtheorem{thm}{Theorem}[section]
    
    \newtheorem{lem}[thm]{Lemma}

    \newcommand{\bea}{\begin{eqnarray}}
    \newcommand{\eea}{\end{eqnarray}}
    \newcommand{\bna}{\begin{eqnarray*}}
    \newcommand{\ena}{\end{eqnarray*}}

    \theoremstyle{definition}

    \numberwithin{equation}{section}

\begin{document}

\title{A pair of Goldbach-Linnik equations in unlike powers of primes and powers of two}
\author{Liqun Hu and Siqi Liu\\
Department of Mathematics, Nanchang University,\\
Nanchang, Jiangxi 330031, P.R. China\\
E-mail: huliqun@ncu.edu.cn and siqiliu463@gmail.com}
\date{\today}
\maketitle
{\bf Abstract}:
In this paper, we show that every pair of large even integers satisfying certain necessary conditions can be expressed as a pair of one prime, one prime square, two prime cubes and 56 powers of 2.\par
\section{Introduction}

In the 1950s, Linnik \cite{LLK1,LLK2} proved that every large even integer $N$ can be expressed as the sum of two primes and a bounded number of powers of 2, i.e.
\begin{equation}\label{1100}
N=p_{1}+p_{2}+2^{v_{1}}+2^{v_{2}}+\cdots +2^{v_{k}}.
\end{equation}
The famous Goldbach conjecture entails that $k=0$. The parameter $k$ has been more accurately determined by many authors.\par
In 1999, Liu, Liu, and Zhan \cite{LLZ} proved that every large even integer $N$ can be expressed as a sum of four squares of primes and a bounded number of powers of 2, i.e.
\begin{equation}\label{1200}
N=p_{1}^{2} +p_{2}^{2}+p_{3}^{2}+p_{4}^{2} +2^{v_{1}}+2^{v_{2}}+\cdots +2^{v_{k}}.
\end{equation}
The best result at present is Zhao \cite{Zhao}, who has achieved a value of $k=39$.\par
In 2001, Liu and Liu \cite{LL} proved that every large even integer $N$ can be expressed as the sum of eight cubes of prime numbers and a bounded number of powers of 2, i.e.
\begin{equation}\label{1300}
N=p_{1}^{3} +p_{2}^{3}+\cdots +p_{8}^{3} +2^{v_{1}}+2^{v_{2}}+\cdots +2^{v_{k}}.
\end{equation}
Researchers have continuously lowered the acceptable value for $k$.
\begin{align*}
k&=358   & \text{(Liu \& L\"{u})\cite{83333}},   \\
k&=341   & \text{(Liu)\cite{LZX2}},  \\
k&=330   & \text{(Platt \& Trudgian)\cite{DT}},\\
k&=204   & \text{(Zhao \& Ge)\cite{ZG}},\\
k&=169   & \text{(Zhao)\cite{ZXDD}},\\
k&=30    & \text{(Zhu)\cite{ZLZL}},\\
k&=28    & \text{(Han \& Liu)\cite{HLLH}}.
\end{align*}\par
In 2011, Liu and L{\"u} \cite{LLUU} considered the hybrid problem of \eqref{1100}, \eqref{1200} and \eqref{1300}, i.e.
\begin{equation}
N=p_{1} +p_{2}^{2}+p_{3}^{3}+p_{4}^{3} +2^{v_{1}}+2^{v_{2}}+\cdots +2^{v_{k}}.
\end{equation}
Then researchers made continuous improvements to the acceptable value of $k$.
\begin{align*}
    k&=161   & \text{( Liu \& L{\"u} \cite{LLUU})},   \\
  k&=156   & \text{( Platt \& Trudgian \cite{DT})},   \\
    k&=16   & \text{(Zhao \cite{Zhao2})},  \\
    k&=15   & \text{(L{\"u} \cite{Lv}  )}.
\end{align*}\par
In 2017, Hu and Yang \cite{HY} first considered the simultaneous representation of pairs of positive even integers satisfying $N_{2} \gg N_{1}>N_{2}$, i.e.
\begin{equation}\label{1500}
\begin{cases}
N_{1} =p_{1} +p_{2}^{2}+p_{3}^{3}+p_{4}^{3} +2^{v_{1}}+2^{v_{2}}+\cdots +2^{v_{k}},\\
N_{2} =p_{5} +p_{6}^{2}+p_{7}^{3}+p_{8}^{3} +2^{v_{1}}+2^{v_{2}}+\cdots +2^{v_{k}},
\end{cases}
\end{equation}
where $k$ is a positive integer. Then researchers made continuous improvements to the acceptable value of $k$.
\begin{align*}
  k&=455   & \text{( Hu \& Yang \cite{HY})},   \\
    k&=302   & \text{( Liu \cite{LUH})},  \\
    k&=187   & \text{( Hu \& Cai \cite{HC}  )}.
\end{align*}\par
The main objective of this paper is to refine this result by giving the following theorem.
\begin{thm}
    For $k=56$, the equations \eqref{1500} are solvable for every pair of large positive even integers $N_{1}$ and $N_{2}$ satisfying $N_{2} \gg N_{1}>N_{2}$.
\end{thm}
In this paper, we make a new estimate of singular series with the help of computer. In addition, we have also made a new estimate of minor arcs by adopting some strategies used in the works of Hu and Cai \cite{HC}, Kong and Liu \cite{KYF} and Hathi \cite{HA}.\par
\textbf{Notation:} Throughout this paper, the $N_{i}$ is employed to represent a large even number satisfying $N_{2} \gg N_{1}>N_{2}$. The symbol $\epsilon $ represents a positive constant, which can be arbitrarily small but may vary across different instances. We use $e(\alpha )$ to denote $e^{2 \pi i \alpha  } $, and $n\sim N$ means $N<n\leq 2N$.

\section{Auxiliary estimation}
In this section we introduce the necessary lemmas and auxiliary estimates to prove Theorem 1.1.\par
In order to apply the circle method, we set
\begin{equation}
\nonumber
P_{i} =N_{i} ^{\frac{1}{9} -2\epsilon }, Q_{i} =N_{i} ^{\frac{8}{9} +\epsilon }, L=\log_{2}{N_{1} }.
\end{equation}
We define the major arcs $\mathcal{M}_i$ and minor arcs $C(\mathcal{M}_i)$ as usual,
\begin{equation}
\nonumber
\mathcal{M}_i=\bigcup_{q\leq{P_i}}
\bigcup_{\substack{1\leq a\leq q \\ (a,q)=1}}{\mathcal{M}_i}(a,q),\ C(\mathcal{M}_i)=\Bigg[\frac{1}{Q_i},1+\frac{1}{Q_i}\Bigg]\setminus{\mathcal{M}_i},
\end{equation}
where $i=1,2$ and
$$
{\mathcal{M}_i}(a,q)=\left\{{\alpha_i}\in[0,1]:{\bigg|{\alpha_i}-\frac{a}{q}\bigg|} \leq {\frac{1}{{q}{Q_i}}}\right\}.
$$
Actually, it can be readily inferred that any two distinct major arcs $\mathcal{M}_1(a,q)$ and $\mathcal{M}_2(a,q)$ are mutually disjoint due to $2P_i\leq Q_i$.
We further define
\begin{equation}
\nonumber
\mathcal{M}=\mathcal{M}_1\times\mathcal{M}_2=\left\{(\alpha_1,\alpha_2):{\alpha_1\in\mathcal{M}_1},{\alpha_2\in\mathcal{M}_2}\right\},
\end{equation}
\begin{equation}
\nonumber
C(\mathcal{M})={{\Bigg[\frac{1}{Q_i},1+\frac{1}{Q_i}\Bigg]}^2}\setminus{\mathcal{M}}.
\end{equation}
Let $\delta =10^{-4} $ and
\begin{equation}\label{2.1}
U_{i} =\left ( \frac{N_{i}}{ 16(1+\delta )}  \right ) ^{\frac{1}{3} },\ \
V_{i} =U_{i}^{\frac{5}{6} } ,
\end{equation}
for $i=1,2$.
We set
\begin{equation}
\nonumber
f(\alpha_i)=\sum_{p\leq {N_i}}^{}(\log_{}p)e({p}\alpha_i),\
g(\alpha_i)=\sum_{p\leq {N_i^{\frac{1}{2} }}}(\log_{}p)e({p^{2} }\alpha_i),
\end{equation}
\begin{equation}
\nonumber
S(\alpha_i)=\sum_{p\sim U_{i} }(\log_{}p)e({p^{3} }\alpha_i),\
T(\alpha_i)=\sum_{p\sim V_{i} }(\log_{}p)e({p^{3} }\alpha_i),
\end{equation}
and set
\begin{equation}
\nonumber
G(\alpha_i)=\sum_{v\leq{L}^{}}e({2^v}\alpha_i),\
{\mathcal{E}_\lambda}=\left\{\alpha_i\in{[0,1]}:|G(\alpha_i)|\geq{\lambda L}\right\},
\end{equation}
for $i=1,2$.\par
Let
\begin{equation}
\nonumber
R(N_1,N_2)=\sum_{\substack{N_{1} =p_{1} +p_{2}^{2}+p_{3}^{3}+p_{4}^{3} +2^{v_{1}}+2^{v_{2}}+\cdots +2^{v_{k}} \\
N_{2} =p_{5} +p_{6}^{2}+p_{7}^{3}+p_{8}^{3} +2^{v_{1}}+2^{v_{2}}+\cdots +2^{v_{k}}\\
p_{1}\leq N_{1},\ p_{2}\leq N_{1}^{\frac{1}{2}},\ p_{3} \sim U_{ 1},\ p_{4} \sim V_{1},\ p_{5}\leq N_{2},\\
p_{6}\leq N_{2}^{\frac{1}{2}},\ p_{7} \sim U_{2},\ p_{8} \sim V_{2},\ v_{j}\leq L(j=1,2,..,k) }}{\log_{}{p_1}}{\log_{}{p_2}}\cdots{\log_{}{p_{8}}}.
\end{equation}
Then $R(N_1,N_2)$ can be represented as
\begin{eqnarray}
&&{\bigg({\iint \limits_{ \mathcal{M}}}}+{\iint \limits_{ C(\mathcal{M})\cap {\mathcal{E}_\lambda}}}+{\iint \limits_{C(\mathcal{M})\setminus{ {\mathcal{E}_\lambda}}}}\bigg) \
\prod_{i=1}^2f(\alpha_i)g(\alpha_i)S^{2}(\alpha_i)T^{2}{(\alpha_i)}{G^k}(\alpha_1+\alpha_2) \notag\\
&& \quad \times{e({-\alpha_1}{N_1})}e({-\alpha_2}{N_2)}d{\alpha_1}\,d{\alpha_2} \notag\\
&&:={R_1}(N_1,N_2)+{R_2}(N_1,N_2)+{R_3}(N_1,N_2).\notag
\end{eqnarray}
The establishment of Theorem 1.1 will be achieved through the estimation of the term $R_{1} (N_{1},N_{2} )$, $R_{2} (N_{1},N_{2} )$ and $R_{3} (N_{1},N_{2} )$. Our goal is to prove that $R(N_1,N_2)>0$ for every pair of large positive even integers $N_1$ and $N_2$ satisfying $N_2\gg N_1> N_2$. Now we state the lemmas required for this paper and make the necessary auxiliary estimates.\par
Let
\begin{equation}
\nonumber
{C_i}(q,a)=\sum_{\substack{m=1\\ (m, q)=1}}^{q}{e\left ( \frac{a{m^i}}{q} \right ) },
\end{equation}
for $i=1,2,3$, and
\begin{equation}
\nonumber
A(n,q)={\frac{\mu (q)}{{\varphi}^4(q)}}{\sum_{\substack{a=1\\ (a, q)=1}}^{q}{C_2}(q,a){C_{3}^{2}}(q,a){e\left ( \frac{-an}{q}  \right ) }}, \
\mathfrak{S}(n)=\sum_{q=1}^{\infty}A(n,q).
\end{equation}
\begin{lem}\label{lemma2.1}
For ${{N_i}/2}\leq {n_i}\leq N_i \; (i=1,2)$, we have
$$
\int \limits_{ \mathcal{M}_i}{f(\alpha_i)}{g(\alpha_i)}{S{(\alpha_i)}}{T{(\alpha_i)}}{e(-{n_i}\alpha_i)}d{\alpha_i}={\frac{1}{{2}\cdot{3^2}}}{\mathfrak{S}({n_i})}{{\mathfrak{J}}}({n_i})+O({N_i}^{\frac{10}{9} }L^{-1}).
$$
where
\begin{equation}
\nonumber
\mathfrak{J}(n_i):=\sum_{\substack{{m_1}+{m_2}+{m_3}+{m_4}={n_i}\\{m_1}\leq {N_i}, \  {m_2}\leq {N_i}\\{U_{i}^{3} }\leq {m_3}\leq 8U_{i}^{3},\ {V_{i}^{3} }\leq {m_4}\leq 8V_{i}^{3}}} {m_2}^{-\frac{1}{2} }{({m_3}{m_4}})^{-\frac{2}{3} }.
\end{equation}
\end{lem}
\begin{proof}
The proof can be found in Liu and L{\"u} \cite[Lemma 2.1]{LLUU}. The detailed discussion can be found in many papers (see \cite{2121} etc.).
\end{proof}
\begin{lem}\label{lemma2.2}
For $(1-\delta )N_i\leq n_i \leq N_i$, we have ${\mathfrak{J}}(n_i)>C_{2} N_i^{\frac{10}{9} }$, with
$$
C_{2}=2.338190371.
$$
\end{lem}
\begin{proof}
The proof of this lemma can be found in Liu and L{\"u} \cite[Lemma 4.2]{LLUU}.
\end{proof}
\begin{lem}\label{lemma2.3}
We have
$$
{(1+A(n_i,5))(1+A(n_i,11))}{\prod_{p\geq17}^{\infty}(1+A(n_i,p))}\geq C :=0.902985.
$$
\end{lem}
\begin{proof}
For $p=5$, $p=11$ and $17\leq p\leq 199$, we can get the result about $\min_{1\leq n\leq p} (1+A(n_i,p))$ directly by using computer calculation.
\begin{equation}
\nonumber
1+A(n_i,5)\ge 0.984375,\  1+A(n_i,11)\ge 0.999000,\dots ,\ 1+A(n_i,199)\ge 0.998903.
\end{equation}
Then
\begin{equation}\label{2100}
(1+A(n_i,5))(1+A(n_i,11))\prod_{17\leq p\leq 199}(1+A(n_i,p))\ge 0.9568859.
\end{equation}
For $p\ge 5$, we quote the estimation formula from \cite[page 133-134]{LLUU}. If $p\equiv 2\ (\bmod 3 )$ and $(a,p)=1$, we have
\begin{equation}\label{2200}
1+A(n_i,p)\ge 1-\frac{\sqrt{p}+1 }{(p-1)^{3} }.
\end{equation}
If $p\equiv 1\ (\bmod 3 )$, we get
\begin{equation}\label{2300}
1+A(n_i,p)\ge 1-\frac{(\sqrt{p}+1)(2\sqrt{p}+1 )^{2}  }{(p-1)^{3} } .
\end{equation}
Combining \eqref{2200}-\eqref{2300} and \cite[(3.14)]{LLUU}, we can deduce from calculation that
\begin{equation}\label{2400}
\begin{split}
\prod_{p>199}(1+A(n_i,p))=\prod_{199<p<10^{6}}(1+A(n_i,p)) \prod_{p\ge 10^{6}}(1+A&(n_i,p)) \\ \ge 0.958892\times 0.984127 \ge 0.943671 .
\end{split}
\end{equation}
Now, we can conclude from \eqref{2100} and \eqref{2400} that
\begin{equation}
\nonumber
{(1+A(n_i,5))(1+A(n_i,11))}{\prod_{p\geq17}^{\infty}(1+A(n_i,p))}\geq C := 0.902985.
\end{equation}
\end{proof}
\begin{lem}\label{lemma2.4}
Let $\Xi (N_i,k)=\left\{n_{i} \ge 2:n_i=N_i-2^ {v_1}- 2^{v_2} -\cdots- 2^ {v_k} \right\}$
with $k\geq2$. Then for $N_1\equiv {N_2}\equiv0\pmod2$, we have
$$
 \sum_{\substack{{n_1}\in \Xi (N_1,k)\\ {n_2}\in \Xi (N_2,k)\\ {n_1}\equiv {n_2}\equiv0\pmod 2 }}{\mathfrak{S}}(n_1){\mathfrak{S}}(n_2) \geq3.261435L^k.
$$
\end{lem}
\begin{proof}
Since $A(n_i,p)$ is multiplicative and $A(n_i,p^k)=0$ for $k\geq2$, we have
$$
\mathfrak{S}(n_i)=\prod_{p=2}^{\infty}(1+A(n_i,p)).
$$
By Lemma 2.2 and $A(n_i,2)=1$ for $n_i \equiv 0 \pmod{2} $, we get
\begin{eqnarray}
   \mathfrak{S}(n_i)
   &\geq &C(1+A(n_i,2))(1+A(n_i,3))(1+A(n_i,7))(1+A(n_i,13))\notag\\
   &=& 2C\prod_{p=3,7,13}^{}(1+A(n_i,p)).\notag
\end{eqnarray}
Let $q={\prod_{p=3,7,13}^{} p}=273$, then
\begin{eqnarray}\label{2500}
&&\sum_{\substack{{n_1}\in \Xi (N_1,k)\\ {n_2}\in \Xi (N_2,k)\\ {n_1}\equiv {n_2}\equiv0\pmod2 }}{\mathfrak{S}}(n_1){\mathfrak{S}}(n_2) \notag\\
&\geq&(2C)^2\sum_{\substack{{n_1}\in \Xi (N_1,k)\\ {n_2}\in \Xi (N_2,k)\\ {n_1}\equiv {n_2}\equiv0\pmod2 }}{\prod_{p=3,7,13}^{}(1+A(n_1,p))}{\prod_{p=3,7,13}^{}(1+A(n_2,p))}\notag
\end{eqnarray}
\begin{eqnarray}\label{2501}
&&\geq(2C)^2{\sum_{\substack {1\leq j\leq q}}} \sum_{\substack{{n_1}\in \Xi (N_1,k)\\ {n_2}\in \Xi (N_2,k)\\ {n_1}\equiv {n_2}\equiv0\pmod2 \\{n_1}\equiv {n_2}\equiv j\pmod q  }}{\prod_{p=3,7,13}^{}(1+A(n_1,p))}{\prod_{p=3,7,13}^{}(1+A(n_2,p))}\notag\\
&&\geq(2C)^2{\sum_{\substack {1\leq j\leq q}}}{\prod_{p=3,7,13}^{}(1+A(j,p))}{\prod_{p=3,7,13}^{}(1+A(j,p))} \sum_{\substack{{n_1}\in \Xi (N_1,k)\\ {n_2}\in \Xi (N_2,k)\\ {n_1}\equiv {n_2}\equiv0\pmod2 \\{n_1}\equiv {n_2}\equiv j\pmod q }} 1 \notag\\
&&\geq(2C)^2{\sum_{\substack {1\leq j\leq q}}}{\prod_{p=3,7,13}^{}(1+A(j,p))^2} \sum_{\substack{{n_1}\in \Xi (N_1,k) \\{n_1}\equiv 0\pmod 2 \\{n_1}\equiv j\pmod q }} 1.
\end{eqnarray}
Let $S$ denote the innermost sum in \eqref{2500}. Since ${N_1}\equiv 0\pmod 2$, we have
\begin{eqnarray}
    S&=&\sum_{\substack{{n_1}\in \Xi (N_1,k) \\{n_1}\equiv 0\pmod 2 \\{n_1}\equiv j\pmod q }} 1\notag \\
    &=&\sum_{\substack{1\leq {v_1,\cdots,v_k}\leq L \\{{N_1}-2^{v_1}-\cdots-2^{v_k}}\equiv 0\pmod 2 \\{{N_1}-2^{v_1}-\cdots-2^{v_k}}\equiv j\pmod q }} 1\\
    &=&\sum_{\substack{1\leq {v_1,\cdots,v_k}\leq L \\{2^{v_1}+\cdots+2^{v_k}}\equiv {{N_1}-j}\pmod q }} 1.
    \notag
\end{eqnarray}
Let $\rho(q)$ denote the smallest positive integer $\rho$ such that $2^{\rho}\equiv 1\pmod q $. Thus
\begin{eqnarray}\label{2600}
    S&=&{{\bigg(\frac{L}{\rho(q)}+O(1)\bigg)}^k}{\sum_{\substack{1\leq {v_1,\cdots,v_k}\leq {\rho(q)} \\{2^{v_1}+\cdots+2^{v_k}}\equiv {{N_1}-j}\pmod q }} 1}\notag \\
    &=&{{\bigg(\frac{L}{\rho(q)}+O(1)\bigg)}^k}{\frac{1}{q}}{\sum_{r=1}^{q}e\bigg({\frac{r(j-{N_1})}{q}}\bigg)}{\bigg(\sum_{\substack{1\leq v\leq{\rho(q)}}}e{(\frac{r2^v}{q})}\bigg)}^k.
\end{eqnarray}
Note that $q=273$, then we can get ${\rho(q)}=12$. Write $f(r)=\bigg|\sum_{\substack{1\leq v\leq{\rho(q)}}}e{(\frac{r2^v}{q})}\bigg|$.
With the help of a computer, it is easy to check that
\begin{eqnarray}\label{2700}
    \max_{\substack{1\leq r\leq q-1}}f(r)=f(91)=6,\ f(q)={\rho(q)}=12.
\end{eqnarray}
From \eqref{2600}-\eqref{2700}, we have
\begin{eqnarray}
    S&\geq& {{\bigg(\frac{L}{\rho(q)}+O(1)\bigg)}^k}{\frac{1}{q}}{\bigg({\rho^k(q)}-(q-1){\bigg(\max_{\substack{1\leq r\leq q-1}}f(r)\bigg)^k}\bigg)}\notag \\
    &\geq&{\frac{L^k}{q}}{{\bigg({1-(q-1){\bigg({\frac{\max_{\substack{1\leq r\leq q-1}}f(r)}{\rho(q)}}}\bigg)^k}\bigg)}}+O(L^{k-1})\notag \\
    &\geq &{\frac{L^k}{273}}{\bigg(1-272\times{\bigg({\frac{1}{2}}\bigg)^{27}}\bigg)}+O(L^{k-1})\notag \\
    &\geq& 0.0036629L^k+O(L^{k-1}),
    \notag
\end{eqnarray}
where the bound $k\geq 27$ is used.\par
So we have
$$
   {\sum_{\substack{{n_1}\in \Xi (N_1,k) \\{n_1}\equiv 0\pmod 2 \\{n_1}\equiv j\pmod q }} 1}\geq  0.0036629L^k+ O{(L^{k-1})}.
$$
Noting that
\begin{eqnarray}
\sum_{j=1}^{p}(1+A(j,p))^2
&=&p+2{\sum_{j=1}^{p}A(j,p)}+\sum_{j=1}^{p}A(j,p)^2\notag\\
&=&p+\sum_{j=1}^{p}A(j,p)^2\notag\\
&\geq& p.
\notag
\end{eqnarray}
Therefore
\begin{eqnarray}
&&\sum_{\substack{{n_1}\in \Xi (N_1,k)\\ {n_2}\in \Xi (N_2,k)\\ {n_1}\equiv {n_2}\equiv0\pmod2 }}{\mathfrak{S}}(n_1){\mathfrak{S}}(n_2) \notag\\
&\geq&  {(2C)^2}{\sum_{j=1}^{p}}{\prod_{p=3,7,13}^{}(1+A(j,p))^2}{0.0036629L^k}+ O{(L^{k-1})}\notag \\
&\geq& {(2C)^2}{\prod_{p=3,7,13}^{}{\sum_{j=1}^{p}}(1+A(j,p))^2}{0.0036629L^k}+ O{(L^{k-1})} \notag \\
&\geq&3.261435L^k+O{(L^{k-1})}.
\notag
\end{eqnarray}
\end{proof}
\begin{lem}\label{lemma2.5}
For every large pair of positive even integers $N_1$ and $N_2$ satisfying $N_2\gg N_1> N_2$, we have
$$
{R_1}(N_1,N_2) \geq
0.055033{{N_1}^{\frac{10}{9} }{N_2}^{\frac{10}{9} }}L^{k}+O({{N_1}^{\frac{10}{9} }{N_2}^{\frac{10}{9} }}).
$$
\end{lem}
\begin{proof}
By Lemma 2.1, 2.2 and 2.4, we are easy to get
\bna
& & {R_1}(N_1,N_2) \notag \\
&=&{\iint \limits_{ \mathcal{M}}}
\prod_{i=1}^2{f(\alpha_i)}{g(\alpha_i)}{S{(\alpha_i)}}{T{(\alpha_i)}}{G^k}(\alpha_1+\alpha_2){e({-\alpha_i}{N_i})}
d{\alpha_1}\,d{\alpha_2} \notag \\
&=&{{\frac{1}{2^2\cdot3^4}}}\sum\limits_{\substack{{n_1}\in \Xi (N_1,k)\\ {n_2}\in \Xi (N_2,k)}}\bigg(\mathfrak{S}(n_1){{\mathfrak{J}}}(n_1)+O({N_1}^{\frac{10}{9} }L^{-1})\bigg){\bigg(\mathfrak{S}(n_2){{\mathfrak{J}}}(n_2)+O({N_2}^{\frac{10}{9} }L^{-1})\bigg)}\notag \\
&\geq&{{\frac{1}{2^2\cdot3^4}}}\sum\limits_{\substack{{n_1}\in \Xi (N_1,k)\\ {n_2}\in \Xi (N_2,k)}}{\bigg({\mathfrak{S}(n_1)}{\mathfrak{S}(n_2)}{\mathfrak{J}}(n_1){\mathfrak{J}}(n_2)\bigg)}+O({N_1}^{\frac{10}{9} }{N_2}^{\frac{10}{9} }L^{k-1})\notag \\
&\geq& {0.055033}{{N_1}^{\frac{10}{9} }
{N_2}^{\frac{10}{9} }}L^{k}+O{({N_1}^{\frac{10}{9} }{N_2}^{\frac{10}{9} }L^{k-1})}.
\ena
\end{proof}
\begin{lem}\label{lemma2.6}
Let $g(\alpha_i)$ and $T(\alpha_i)$ be defined as before, we have
$$
(i)\ \ {\int_{0}^{1} |{g^2(\alpha_i)}{S^4{(\alpha_i)}}|}d{\alpha_i}\ll {N_i}^{\frac{4}{3} },
$$
$$
(ii)\ \ {\int_{0}^{1} |{g^2(\alpha_i)}{T^4{(\alpha_i)}}|}d{\alpha_i}\ll {N_i}^{\frac{10}{9} }.
$$
\end{lem}
\begin{proof}
The proof of $(i)$ is from \cite[Proof of Lemma 4.3]{Zhao2} with $k=3$. The proof of $(ii)$ can be obtained by using $N_{i}^{\frac{5}{18} }$ instead of $N_{i}^{\frac{1}{3} } $ in $S(\alpha _{i})$.
\end{proof}
\begin{lem}\label{lemma2.7}
 We have
\begin{equation}
\nonumber
\begin{split}
 (i)\ \  {\int \limits_{C(\mathcal{M}_i)}|{g^2(\alpha_i)}{S^5(\alpha_i)}|d{\alpha_i}}\ll {N_i}^{\frac{13}{8} }, \\
 (ii)\ \  {\int \limits_{C(\mathcal{M}_i)}|{g^2(\alpha_i)}{T^5(\alpha_i)}|d{\alpha_i}}\ll {N_i}^{\frac{65}{48} }.
\end{split}
\end{equation}
\end{lem}
\begin{proof}
The proof of $(i)$ is from \cite[Lemma 4.8]{ZHUL}. The proof of $(ii)$ can be obtained by using $N_{i}^{\frac{5}{18} }$ instead of $N_{i}^{\frac{1}{3} } $ in $S(\alpha _{i})$.
\end{proof}
\begin{lem}\label{lemma2.8}
    Let ${\mathcal{E}_\lambda}$ be defined before, we have
    $$
    meas({\mathcal{E}_\lambda})\ll {N_i}^{-E(\lambda)} \ with \  E(0.8512)>\frac{25}{36}+10^{-10}.
    $$
\end{lem}
\begin{proof}
To obtain explicit for $\lambda $. We can see \cite{5} and \cite{17}. You can download the PARI/GP source code of they program at www.math.unipd.itlanguasc/Pintz Ruzsa Method.html.
\end{proof}
\begin{lem}\label{lemma 2.9}
For every pair of large positive even integers $N_1$ and $N_2$ satisfying $N_2\gg N_1> N_2$, we have
$$
{R_2}(N_1,N_2) \ll {{N_1}^{\frac{10}{9} }{N_2}^{\frac{10}{9} }}L^{k-1}\ for \ \lambda=0.8512.
$$
\end{lem}
\begin{proof}
According to the definition of $C(\mathcal{M})$, we have
\begin{align*}
    C(\mathcal{M})&\subset \left\{(\alpha_1,\alpha_2):{\alpha_1\in C(\mathcal{M}_1}),{\alpha_2\in [0,1]}\right\} \\
&\cup\left\{(\alpha_1,\alpha_2):{\alpha_1\in [0,1],{\alpha_2\in C(\mathcal{M}_2})}\right\}.
\end{align*}
Then
\bna
& & {R_2}(N_1,N_2) \notag \\
&=&{\iint \limits_{ C(\mathcal{M})\cap {\mathcal{E}_\lambda}}}
\prod_{i=1}^2{f(\alpha_i)}{g(\alpha_i)}{S{(\alpha_i)}{T(\alpha_i)}}
 {G^k}(\alpha_1+\alpha_2)
 {e({-\alpha_i}{N_i})}d{\alpha_1}\,d{\alpha_2} \notag \\
&\ll& {L^k}\bigg({\iint \limits_{\substack{(\alpha_1,\alpha_2)\in {C(\mathcal{M}_1})\times [0,1]\\|G(\alpha_1+\alpha_2)|\geq{\lambda L}}}}+ {\iint \limits_{\substack{(\alpha_1,\alpha_2)\in {C(\mathcal{M}_2})\times [0,1]\\|G(\alpha_1+\alpha_2)|\geq{\lambda L}}}}\bigg)\bigg|\prod_{i=1}^2{f(\alpha_i)}{g(\alpha_i)}{S{(\alpha_i)}{T(\alpha_i)}}\bigg|
d{\alpha_1}\, d{\alpha_2} \notag \\
&:=& {L^k}(I_1+I_2),
\ena
where we have employed the elementary bound of $G(\alpha_1+\alpha_2)$. By Cauchy’s inequality, we get
\bna
{I_1}&\ll& {\iint \limits_{\substack{(\alpha_1,\alpha_2)\in {C(\mathcal{M}_1})\times [0,1]\\|G(\alpha_1+\alpha_2)|\geq{\lambda L}}}}\bigg|\prod_{i=1}^2 {f(\alpha_i)}{g(\alpha_i)}{S{(\alpha_i)}{T(\alpha_i)}}
\bigg| d{\alpha_1}\,d{\alpha_2} \notag    \\
&\ll&\int_{0}^{1}  |{f(\alpha_2)}{g(\alpha_2)}{S(\alpha_2)}{T(\alpha_2)}|d{\alpha_2} \int \limits_{\substack{ C({\mathcal{M}_1)}\\|G(\alpha_1+\alpha_2)|\geq{\lambda L}}}|{f(\alpha_1)}{g(\alpha_1)}{S{(\alpha_1)}{T(\alpha_1)}}|d {\alpha_1}\\
&\ll& \left ( \int_{0}^{1}  |{f(\alpha_2)}|^{2} d{\alpha_2} \right )^{\frac{1}{2} }
\left ( \int_{0}^{1}  |{g(\alpha_2)}|^{2}|{S(\alpha_2)}|^{2}|{T(\alpha_2)}|^{2} d{\alpha_2} \right )^{\frac{1}{2} }\\
&&\times
\left ( \int \limits_{\substack{C({\mathcal{M}_1)}\\|G(\alpha_1+\alpha_2)|\geq{\lambda L}}}  |{f(\alpha_1)}|^{2} d{\alpha_1} \right )^{\frac{1}{2} }
\left ( \int \limits_{\substack{C({\mathcal{M}_1)}\\|G(\alpha_1+\alpha_2)|\geq{\lambda L}}}  {|g(\alpha_1)}|^{2}|{S(\alpha_1)}|^{2}|{T(\alpha_1)}|^{2} d{\alpha_1} \right )^{\frac{1}{2} }\\
&:=&I_{11} \times I_{12}.
\ena
Applying the H\"{o}lder's inequality, we have
\bna
I_{12}&=&\left ( \int \limits_{\substack{C({\mathcal{M}_1)}\\|G(\alpha_1+\alpha_2)|\geq{\lambda L}}}  |{f(\alpha _{1} )}|^{2} d{\alpha_1} \right )^{\frac{1}{2} }
\left ( \int \limits_{\substack{C({\mathcal{M}_1)}\\|G(\alpha_1+\alpha_2)|\geq{\lambda L}}}  {|g(\alpha _{1} )}|^{2}|{S(\alpha _{1} )}|^{2}|{T(\alpha _{1} )}|^{2} d{\alpha_1} \right )^{\frac{1}{2} } \\
&\ll& \left ( \int_{0}^{1}  |{f(\alpha _{1} )}|^{2} d{\alpha_1} \right )^{\frac{1}{2} }
\left ( \int \limits_{\substack{C({\mathcal{M}_1)}\\|G(\alpha_1+\alpha_2)|\geq{\lambda L}}}  {|g(\alpha _{1} )}|^{2}|{S(\alpha _{1} )}|^{4} d{\alpha_1} \right )^{\frac{1}{4} } \\
&&\times \left ( \int \limits_{\substack{C({\mathcal{M}_1)}\\|G(\alpha_1+\alpha_2)|\geq{\lambda L}}}  {|g(\alpha _{1} )}|^{2}|{T(\alpha _{1} )}|^{4} d{\alpha_1} \right )^{\frac{1}{4} }.
\ena
By the simple orthogonality, we have
\begin{equation}\label{2.100}
\int_{0}^{1}  |{f(\alpha _{1} )}|^{2} d{\alpha_1} \ll N_{1}^{1+\epsilon}.
\end{equation}
Applying the H\"{o}lder's inequality and using (i) of Lemmas 2.6-2.7, we obtain
\bna
&&\int \limits_{\substack{C({\mathcal{M}_1)}\\|G(\alpha_1+\alpha_2)|\geq{\lambda L}}}  {|g(\alpha _{1} )}|^{2}|{S(\alpha _{1} )}|^{4} d{\alpha_1}  \notag    \\
&\ll&\left ( \int \limits_{C({\mathcal{M}_1)}}  {|g(\alpha _{1} )}|^{2}|{S(\alpha _{1} )}|^{4} d{\alpha_1}  \right )^{\frac{1}{6} }
\left ( \int \limits_{C({\mathcal{M}_1)}}  {|g(\alpha _{1} )}|^{2}|{S(\alpha _{1} )}|^{5} d{\alpha_1}  \right ) ^{\frac{2}{3} } \\
&\ \ & \times \left ( \int_{0}^{1}  {|g(\alpha _{1} )}|^{4} d{\alpha_1}  \right ) ^{\frac{1}{12} }
\left ( \int_{\mathcal{E}(\lambda )} 1 d{\alpha_1}  \right ) ^{\frac{1}{12} }\\
&\ll&N_{1}^{\frac{25}{18} } N_{1}^{-\frac{1}{12}E(\lambda ) },
\ena
where we have used the Rieger's result \cite{R}
$$
{\int_{0}^{1}|{g^4(\alpha_i)}|d{\alpha_i}}\ll {N_i}{\log^2{N_i}}.
$$
Similarly, we have
$$
\int \limits_{\substack{C({\mathcal{M}_1)}\\|G(\alpha_1+\alpha_2)|\geq{\lambda L}}}  {|g(\alpha _{1} )}|^{2}|{T(\alpha _{1} )}|^{4} d{\alpha_1}
\ll N_{1}^{\frac{253}{216} }N_{1}^{-\frac{1}{12}E(\lambda ) }.
$$
Then
$$
I_{12} \ll
\left ( N_{1}^{1+\epsilon }   \right ) ^{\frac{1}{2}}
\left ( N_{1}^{\frac{25}{18}} N_{1}^{-\frac{1}{12} E(\lambda )} \right )^{\frac{1}{4} }
\left ( N_{1}^{\frac{253}{216} } N_{1}^{-\frac{1}{12}E(\lambda ) } \right ) ^{\frac{1}{4} }
\ll N_{1}^{\frac{10}{9}+\frac{25}{864}  } N_{1}^{-\frac{1}{24} E(\lambda )}.
$$
Applying the H\"{o}lder's inequality, by Lemma 2.6 and \eqref{2.100} we have
\bna
I_{11}&=&\left ( \int_{0}^{1}  |{f(\alpha_2)}|^{2} d{\alpha_2} \right )^{\frac{1}{2} }
\left ( \int_{0}^{1}  |{g(\alpha_2)}|^{2}{|S(\alpha_2)}|^{2}{|T(\alpha_2)}|^{2} d{\alpha_2} \right )^{\frac{1}{2} }  \notag    \\
&\ll& \left ( \int_{0}^{1}  {|f(\alpha _{2} )}|^{2} d{\alpha_2}  \right ) ^{\frac{1}{2} }
\left ( \int_{0}^{1}   {|g(\alpha _{2} )}|^{2}|{S(\alpha _{2} )}|^{4} d{\alpha_2}  \right )^{\frac{1}{4} } \\
&&\times \left ( \int_{0}^{1}  {|g(\alpha _{2} )}|^{2}|{T(\alpha _{2} )}|^{4} d{\alpha_2}  \right ) ^{\frac{1}{4} }   \\
&\ll& N_{2}^{\frac{10}{9} }.
\ena
Combining with $I_{11}$ and $I_{12}$, by Lemma 2.8 we get
$$
I_{1}\ll N_{1}^{\frac{10}{9}+\frac{25}{864}  } N_{1}^{-\frac{1}{24} E(\lambda )}N_{2}^{\frac{10}{9}}\ll N_{1}^{\frac{10}{9}-\epsilon  }N_{2}^{\frac{10}{9}},
$$
since $N_{2} \gg N_{1}>N_{2}$. Similarly,
$$
I_{2}\ll N_{2}^{\frac{10}{9}-\epsilon  }N_{1}^{\frac{10}{9}}.
$$
Then
$$
{R_2}(N_1,N_2) \ll  {{N_1}^\frac{10}{9}}{{N_2}^\frac{10}{9}}L^{k-1}.
$$
\end{proof}
\begin{lem}\label{lemma 2.10}
We have
$$
\iint\limits_{(\alpha _{1},\alpha _{2} )\in [0,1]^{2} }
\left | f^{2}(\alpha _{1})f^{2}(\alpha _{2}) G^{4} (\alpha _{1}+\alpha _{2}) \right |
\mathrm{d}\alpha _{1}\mathrm{d}\alpha _{2}\leq 305.8869N_{1}N_{2}L^{4} .
$$
\end{lem}
\begin{proof}
The proof of this lemma is shown in \cite[Lemma 2.3]{KYF}.
\end{proof}
\begin{lem}\label{lemma 2.11}
We have
\bna
\iint\limits_{(\alpha _{1},\alpha _{2} )\in [0,1]^{2} }&&
\left | g^{4}(\alpha _{1})g^{4}(\alpha _{2}) G^{22} (\alpha _{1}+\alpha _{2}) \right |
\mathrm{d}\alpha _{1}\mathrm{d}\alpha _{2}  \notag \\
&&\leq181132.16N_{1}N_{2}L^{22}+O(N_{1}N_{2}L^{22}).
\ena
\end{lem}
\begin{proof}
Following the proof of \cite[Lemma 3.2]{ZLL3} and \cite[Lemma 2.2]{HA}, we address the cases $h\ne 0$ and $h=0$ separately to derive
\bna
&&\iint\limits_{(\alpha _{1},\alpha _{2} )\in [0,1]^{2} } \left | g^{4}(\alpha _{1})g^{4}(\alpha _{2})G^{22} (\alpha _{1}+\alpha _{2}) \right |
\mathrm{d}\alpha _{1}\mathrm{d}\alpha _{2} \\
&\leq& (8\cdot 11)^{2} N_{1} N_{2} \sum_{h\ne 0}r_{11} (h)\mathbf{S} ^{2} (h)+O(N_{1}N_{2}L^{22}).
\ena
where
$$
r_{11} (h)=\sum_{\substack{4\leq v_{j}u_{j} \leq L\\\sum_{j=1}^{11} (2^{v_{j}}-2^{u_{j}} )=h} }1,\ \
$$
$$
\mathbf{B} (p,h)=\sum_{\substack{a=1\\(a,p)=1} }^{q} \left | C_{2} (p,a) \right | ^{4} e\left ( \frac{ah}{p}  \right ) ,\ \ \ \
\mathbf{S}(h)=\prod_{p>2}\left ( 1+\frac{\mathbf{B} (p,h)}{(p-1)^{4} }  \right ) .
$$
Refer to the proof process in \cite[Lemma 4.3]{ZLL3} and \cite[Lemma 2.6]{HC}, we get
$$
\sum_{h\ne 0}r_{11} (h)\mathbf{S} ^{2} (h)\leq 23.39L^{22},
$$
then
\bna
&&\iint\limits_{(\alpha _{1},\alpha _{2} )\in [0,1]^{2} }
\left | g^{4}(\alpha _{1})g^{4}(\alpha _{2})G^{22} (\alpha _{1}+\alpha _{2}) \right |
\mathrm{d}\alpha _{1}\mathrm{d}\alpha _{2}\\
&\leq& (8\cdot 11)^{2} N_{1} N_{2}\cdot 23.39L^{22}+O(N_{1}N_{2}L^{22})  \notag    \\
&\leq&181132.16N_{1}N_{2}L^{22}+O(N_{1}N_{2}L^{22}).
\ena
\end{proof}
\begin{lem}\label{lemma 2.12}
We have
$$
\sum_{\substack{m_i \sim U_i\\(m_i,P(z))=1} }R(m_i) \leq
100551.95119U_iV_i^{4}L^{-8}.
$$
where $R(m_i)$ denotes the number of solutions of the equations
$$
m_i^{3} -p_{9}^{3} +p_{10}^{3} -p_{11}^{3} +p_{12}^{3}-p_{13}^{3}+p_{14}^{3}-p_{15}^{3}=0,
$$
with $U_i\leq p_{9} ,p_{10} ,p_{11} \leq 2U_i,\ V_i\leq p_{12}, p_{13}, p_{14} ,p_{15} \leq2V_i$ and $P(z)=\prod_{11\leq p\leq z}p$ with $z=N_i^{11(1-\epsilon )/180}$.
\end{lem}
\begin{proof}
We define
$$ \mathfrak{S}_1=\sum_{q=1}^\infty T_d(q),  $$
where
$$T_d(q)=\sum_{\substack{a=1\\(a,q)=1}}^q \frac{S(q,ad^3)C(p,a)^3\overline{{C(p,a)}^4}}{q \varphi(q)^{7}}, $$ and
$$S(q,a)=\sum_{m=1}^q e\bigg(\frac{am^3}{q}\bigg),\quad C(q,a)=\sum_{\substack{m=1\\(m,q)=1}}^q e\bigg(\frac{am^3}{q}\bigg). $$
According to \cite [Lemma 9.1] {7} and \cite[lemma 4.1]{ZHULI}, we obtain
$$
\sum_{\substack{m_i \sim U_i\\(m_i,P(z))=1} }R(m_i)\leq (1+\epsilon )e^{\gamma } J_i\mathfrak{S} _{1} W(z),
$$
where $\gamma =0.577215649\cdots $  is Euler’s constant, and $$W(z)=\prod_{11\leq p\leq N_i^{\frac{11(1-\epsilon )}{180}}} \bigg(1-\frac{w(p)}{p}\bigg),$$
$$\mathfrak{S}_1=(1+T_1(3)+T_1(9)) \prod_{p\neq3}  (1+T_1(p)), \quad w(p)=\frac{1+T_p(p)}{1+T_1(p)}.$$
Following \cite[Lemma 4]{ESP}, we get
\begin{equation}\label{216}
J_i\leq 440.62U_iV_i^{4}L^{-7} .
\end{equation}
Next we will perform more accurate calculations on $\mathfrak{S}_1 W(z)$.  By calculation,  we have
$$1-\frac{w(p)}{p}=\bigg(1-\frac{1}{p}\bigg)\bigg(1-\frac{T_p(p)-T_1(p)}{(p-1)(1+T_1(p))}\bigg).$$
Then by \cite[page 367]{ESP}
\begin{align}\label{110}
  \prod_ {11\leq p\leq N^{\frac{11(1-\epsilon )}{180}}} \bigg(1-\frac{w(p)}{p}\bigg) \leq \frac{40.197}{\log N} \prod_{p\geq11}\bigg(1-\frac{T_p(p)-T_1(p)}{(p-1)(1+T_1(p))}\bigg).
\end{align}
If $p\mid a$, then $S(p,a)=p$. If $p\nmid a$, then by Weil estimates we have $|S(p,a)|\leq 2\sqrt{p}$. Also by \cite[(4.6)]{RX}, we obtain
$$ \sum_{a=1}^{p-1}|C(p,a)^3\overline{{C(p,a)}^4} | \leq (2\sqrt{p}+1)^5(p-1)(2p+1). $$
Hence
\begin{align}\label{130}
|T_p(p)|\leq\frac{(2\sqrt{p}+1)^5(2p+1)}{(p-1)^6}   ,\quad  |T_1(p)|\leq \frac{2(2\sqrt{p}+1)^5(2p+1)}{\sqrt{p}(p-1)^6}.
\end{align}
Then for $p\geq 13$.
 \begin{align}\label{120}
   \frac{T_p(p)-T_1(p)}{(p-1)(1+T_1(p))}&\leq   \frac{\frac{(2\sqrt{p}+1)^5(2p+1)}{(p-1)^6}
   +\frac{2(2\sqrt{p}+1)^5(2p+1)}{\sqrt{p}(p-1)^6}}{ {(p-1)\bigg(1-\frac{2(2\sqrt{p}+1)^5(2p+1)}{\sqrt{p}(p-1)^6}\bigg)}}   \notag \\  &\leq1.2304\frac{(2\sqrt{p}+1)^5(2p+1)(2+\sqrt{p})}{(p-1)^7\sqrt{p}}.
 \end{align}
Therefore
\begin{align}\label{121}
    \prod_{p>4000}\bigg(1-\frac{T_p(p)-T_1(p)}{(p-1)(1+T_1(p))}\bigg)&\leq \prod_{p>4000}\bigg(1+1.2304\frac{(2\sqrt{p}+1)^5(2p+1)(2+\sqrt{p})}{(p-1)^7\sqrt{p}}\bigg)\notag \\
    &\leq \prod_{p>4000}\bigg(1+\frac{M_1}{p^{\frac{7}{2}}}\bigg)
    \leq \prod_{p>4000}\bigg(1+\frac{1}{p^{\frac{7}{2}}}\bigg)^{M_1}\notag \\
    &\leq \bigg(\frac{\zeta(\frac{7}{2})}{\zeta(7)} \prod_{p<4000}\bigg(1+\frac{1}{p^{\frac{7}{2}}}\bigg)^{-1}   \bigg )^{M_1}.
\end{align}
where $M_1=1.2304(1+\frac{2}{\sqrt{4000}})(2+\frac{1}{\sqrt{4000}})^5(2+\frac{1}{4000})(1-\frac{1}{4000})^{-7}=84.6567.$
Combining \eqref{120} and \eqref{121}, we can get
\begin{align*}
 &\prod_{p\geq11}\bigg(1-\frac{T_p(p)-T_1(p)}{(p-1)(1+T_1(p))}\bigg) \\
 &\leq  \prod_{11\leq p \leq 500}\bigg(1-\frac{T_p(p)-T_1(p)}{(p-1)(1+T_1(p))}\bigg) \bigg(\frac{\zeta(\frac{7}{2})}{\zeta(7)} \prod_{p<4000}\bigg(1+\frac{1}{p^{\frac{7}{2}}}\bigg)^{-1}    \bigg)^{M_1}\\
& \quad \times \prod_{500< p \leq 4000}\bigg(1+1.2304\frac{(2\sqrt{p}+1)^5(2p+1)(2+\sqrt{p})}{(p-1)^7\sqrt{p}}\bigg)\\
&\leq 1.0294133 \cdot (1+3.85\cdot 10^{-9}) \cdot (1+9.64\cdot 10^{-7})\\
&\leq 1.0294143.
\end{align*}
Plugging this value into \eqref{110} we obtain
\begin{align}\label{122}
   W(z)= \prod_{11 \leq p \leq N_i^ {\frac{11(1-\epsilon )}{180}}}\bigg(1-\frac{w(p)}{p}\bigg)\leq \frac{41.379367}{\log N_i}.
\end{align}
Similarly, we can estimate $\mathfrak{S}_1$.
\begin{align}\label{123333}
\prod_{ p \geq 4000}(1+T_1(p))&\leq
    \prod_{p\geq 4000}\bigg(1+\frac{2(2\sqrt{p}+1)^5(2p+1)}{\sqrt{p}(p-1)^6}\bigg) \notag\\
    &\leq \prod_{p \geq 4000}\bigg(1+\frac{M_2}{p^3}\bigg)\leq \prod_{p \geq 4000}\bigg(1+\frac{1}{p^3}\bigg)^{M_2}  \notag\\
    &\leq \bigg(\frac{\zeta(3)}{\zeta(6)}\prod_{p<4000}\bigg(1+\frac{1}{p^3}\bigg)^{-1}\bigg)^{M_2}.
\end{align}
where $M_2=2(2+\frac{1}{\sqrt{4000}})^5(2+\frac{1}{4000})(1-\frac{1}{4000})^{-6}=133.3569.$ Therefore by \eqref{130} and \eqref{121}, we get
\begin{align}\label{124}
    \mathfrak{S}_1&\leq (1+T_1(3)+T_1(9))\prod_{5\leq p \leq 500}(1+T_1(p))\cdot \bigg(\frac{\zeta(3)}{\zeta(6)}\prod_{p<4000}\bigg(1+\frac{1}{p^3}\bigg)^{-1}\bigg)^{M_2} \notag \\
   & \quad \times \prod_{500\leq p \leq 4000}\bigg(1+\frac{2(2\sqrt{p}+1)^5(2p+1)}{\sqrt{p}(p-1)^6} \bigg)\notag \\
   &\leq 3.0963\cdot1.00000047413  \cdot 1.00003994288 \leq 3.096427.
\end{align}
Collecting estimates \eqref{216}, \eqref{122} and \eqref{124}, we obtain
\begin{align} \label{bb}
 \sum_{\substack{m_i \sim U_i\\(m_i,P(z))=1}}R(m_i)
 &\leq e^{\gamma } 440.62U_iV_i^{4} L^{-7} \cdot128.1282L^{-1}\notag \\ &\leq 100551.95119U_iV_i^{4} L^{-8}.
 \notag
\end{align}
This proves the Lemma 2.12.
\end{proof}
\begin{lem}\label{lemma 2.13}
We have
  $$ \int_0^1 |S(\alpha_i)^4T(\alpha_i)^4|d\alpha_i \leq 7.390869U_iV_i^{4}. $$
\end{lem}
\begin{proof}
The idea of this proof is similar to \cite[Lemma 4.2]{ZHULI}, with a slight improvement. Considering the number of solutions of the underlying equation, we get
\begin{align*}
  % 计算从0到1的积分
  \int_0^1
  % 积分函数
  |S(\alpha_i)^4T(\alpha_i)^4|
  % 积分变量
  d\alpha_i
  % 不等式
  \leq
  % 对数项
  \log^4(2U_i)\log^4(2V_i)
  % 求和符号和范围
  \sum_{\substack{m_i \sim U_i\\(m_i,P(z))=1}}
  % 求和函数
  R(m_i).
\end{align*}
Since $\frac{\log{}U_i}{\log{}N_i}=\frac{1}{3} +O(\frac{1}{\log{}N_i})$ and $
\frac{\log{}V_i}{\frac{5}{6} \log{}N_i}=\frac{1}{3} +O(\frac{1}{\log{}N_i}).$ By Lemma 2.12, we obtain
\bna
\int_0^1 |S(\alpha_i)^4T(\alpha_i)^4|d\alpha_i &\leq &
100551.95119\log^4(2U_i)\log^4(2V_i)L^{-8}U_iV_i^{4} \\
&\leq& 100551.95119\left(\frac{1}{3} \right ) ^{4}\left (\frac{5}{18}\right)^{4}U_iV_i^{4} \\
&\leq& 7.390869U_iV_i^{4}.
\ena
\end{proof}
\begin{lem} \label{lemma 2.14}
We have
\bna
\iint \limits_{(\alpha _{1},\alpha _{2} )\in [0,1]^{2} }&&
\left | S^{4}(\alpha _{1})T^{4}(\alpha _{1})S^{4}(\alpha _{2})T^{4}(\alpha _{2})  \right |
\mathrm{d}\alpha _{1}\mathrm{d}\alpha _{2}    \notag    \\
&\leq& 54.62495U_{1} U_{2} V_{1}^{4}  V_{2}^{4}.
\ena
\end{lem}
\begin{proof}
By Lemma 2.13, we can lead to the above conclusion.
\end{proof}
\begin{lem}\label{lemma 2.15}
For every pair of large positive even integers $N_1$ and $N_2$ satisfying $N_2\gg N_1> N_2$,
$$
{R_3}(N_1,N_2) \leq 132.42956\lambda^{k-\frac{15}{2}} {{N_1}^\frac{10}{9}}{{N_2}^\frac{10}{9}}L^k.
$$
\end{lem}
\begin{proof}
Combining Lemma 2.10, 2.11 and 2.14 and the definition of $U_{i} , V_i$ in \eqref{2.1}, we get
\bna
&&{R_3}(N_1,N_2) \\
&=& {\iint \limits_{C(\mathcal{M})\setminus{ {\mathcal{E}_\lambda}}}} \prod_{i=1}^2f(\alpha_i)g(\alpha_i)S{(\alpha_i)}T{(\alpha_i)}{G^k}(\alpha_1+\alpha_2)
{e({-\alpha_1}{N_1})} e({-\alpha_2}{N_2})d{\alpha_1}\,d{\alpha_2}\notag\\
&\leq& (\lambda L)^{k-\frac{15}{2}}\left ( \ {\iint \limits_{(\alpha _{1},\alpha _{2} )\in [0,1]^{2} }} |f^2(\alpha_1)f^2(\alpha_2){G^4}(\alpha_1+\alpha_2)
| d{\alpha_1}\,d{\alpha_2} \right ) ^{\frac{1}{2} }\\
&&\times \left (\  {\iint \limits_{(\alpha _{1},\alpha _{1} )\in [0,1]^{2} }} |g^4(\alpha_1)g^4(\alpha_2){G^{22}}(\alpha_1+\alpha_2)
| d{\alpha_1}\,d{\alpha_2} \right ) ^{\frac{1}{4} } \\
&&\times \left (\ {\iint \limits_{(\alpha _{1},\alpha _{2} )\in [0,1]^{2} }} |S^4(\alpha_1)T^4(\alpha_1)S^4(\alpha_2)T^4(\alpha_2)
| d{\alpha_1}\,d{\alpha_2} \right ) ^{\frac{1}{4} } \\
&\leq& (\lambda L)^{k-\frac{15}{2}}
\left ( 305.8869N_{1}N_{2}L^{4} \right ) ^{\frac{1}{2} }
\left ( 181132.16N_{1}N_{2}L^{22}+O(N_{1}N_{2}L^{22}) \right ) ^{\frac{1}{4} }\\
&&\times \left (54.62495U_{1} U_{2} V_{1}^{4}  V_{2}^{4} \right )^{\frac{1}{4} } \\
&\leq& 132.42956\lambda^{k-\frac{15}{2} } {{N_1}^\frac{10}{9}}{{N_2}^\frac{10}{9}}L^k.
\ena
\end{proof}
\section{Proof of Theorem 1.1}
Combining Lemma 2.5, 2.9 and 2.15, we get
\begin{equation}
\nonumber
\begin{split}
{R}(N_1,N_2)=&{R_1}(N_1,N_2)+{R_2}(N_1,N_2)+{R_3}(N_1,N_2)\\
>&(0.055033-132.42956{\lambda}^{k-\frac{15}{2}}){{N_1}^{\frac{10}{9} }{N_2}^{\frac{10}{9}}}L^{k}.
\end{split}
\end{equation}
When $k\geq 56$ and $\lambda =0.8512$, we solve the inequality
$${R}(N_1,N_2)>0.$$
Now, the proof of Theorem 1.1 has been completed.

\section*{Acknowledgements}

This work is supported by Natural Science Foundation of China (Grant Nos. 12361002) and Natural Science Foundation of Jiangxi Province (Grant Nos. 20224BAB201001). The authors would like to express their thanks to the referee for many useful suggestions and comments on the manuscript.

\end{document}